\documentclass[12pt, a4paper]{article}

\usepackage{amsmath}\multlinegap=0pt
\usepackage[latin1]{inputenc}
\usepackage{amsthm}
\usepackage{amssymb}
\usepackage[francais,english]{babel}
\usepackage{color}
\usepackage{graphicx}
\usepackage{tikz}

\numberwithin{equation}{section}
\theoremstyle{plain}
\newtheorem{thm}{Theorem}[section]
\newtheorem{coro}[thm]{Corollary}

\newtheorem{prop}[thm]{Proposition}
\newtheorem{lem}[thm]{Lemma}
\newtheorem{defi}[thm]{Definition}

\theoremstyle{definition}

\theoremstyle{remark}
\newtheorem{rem}[thm]{Remark}

\newcommand\eps{\varepsilon}
\newcommand\F{{\cal F}}
\newcommand\Fr{{\F}^{\rm{rad}}}
\newcommand\EE{{\cal E}}
\newcommand\PP{{\cal P}}
\newcommand\PPa{{\cal P}_{\rm{ac}}(\Omega)}

\newcommand\tilm{{\widetilde{\mu}}}

\newcommand\tilr{{\widetilde{\rho}}}
\newcommand\tilfi{{\widetilde{\varphi}}}
\newcommand\tilz{{\widetilde{z}}}
\newcommand\tilT{{\widetilde{T}}}
\newcommand\tilb{{\widetilde{\beta}}}

\newcommand{\R}{\mathbb{R}}
\newcommand{\BV}{\mathrm{BV}}
\newcommand{\id}{\mathrm{id}}
\newcommand{\Per}{\mathrm{Per}}
\newcommand{\dive}{\mathrm{div}}
\newcommand{\spt}{\mathrm{spt}}
\newcommand{\argmin}{\mathrm{argmin}}
\newcommand{\Omb}{\overline{\Omega}}
\newcommand{\ovmu}{\overline{\mu}}
\newcommand\pref[1]{(\ref{#1})}

\begin{document}
\title{On the total variation Wasserstein gradient flow and the TV-JKO scheme} 
\author{Guillaume   {\sc Carlier}\thanks{Ceremade, UMR CNRS 7534,  Universit\'e Paris Dauphine, Pl. de Lattre de Tassigny, 75775, Paris Cedex 16, France, and MOKAPLAN, INRIA-Paris, E-mail: carlier@ceremade.dauphine.fr }     \and 
 Clarice  {\sc Poon }\thanks{Centre for Mathematical Sciences, University of Cambridge, Wilberforce Rd, Cambridge CB3 0WA, United Kingdom, 
  Email: C.M.H.S.Poon@maths.cam.ac.uk}  
}

\maketitle

\begin{abstract}
We study the JKO scheme for the total variation, characterize the optimizers, prove some of their qualitative properties (in particular a form of maximum principle and in some cases, a minimum principle as well).  Finally,  we establish a convergence result as the time step goes to zero to a solution of a fourth-order nonlinear evolution equation, under the additional assumption that the density remains bounded away from zero. This lower bound is shown in dimension one and in the radially symmetric case. 

\end{abstract}


\medskip

\textbf{MS Classification:}  35G31, 49N15. 

\section{Introduction}\label{sec-intro}

Variational schemes based on total variation are extremely popular in image processing for denoising purposes, in particular the seminal work of Rudin, Osher and Fatemi \cite{Rof} has been extremely influential and is still the object of an intense stream of research, see \cite{chambolle2016geometric} and the references therein. Continuous-time counterparts are well-known to be related to the $L^2$ gradient flow of the total variation, see Bellettini, Caselles and Novaga \cite{bellettini2002total} and the mean-curvature flow, see Evans and Spruck \cite{evansspruck}. The gradient flow of the total variation for other Hilbertian structures may be natural as well and in particular the $H^{-1}$  case, leads to a singular fourth-order evolution equation studied by Giga  and Giga \cite{giga2010}, Giga, Kuroda and  Matsuoka  \cite{giga2014}. In the present work,  we consider another metric, namely the Wasserstein one.

\smallskip

Given an open subset $\Omega$ of $\R^d$ and $\rho \in L^1(\Omega)$, recall that the total variation of $\rho$ is given by
\begin{equation}\label{deftv}
J(\rho):=\sup \Big\{ \int_{\Omega} \dive(z) \rho \; : \; z\in C_c^1(\Omega), \; \Vert z \Vert_{L^\infty} \le 1  \Big\}
\end{equation}
and $\BV(\Omega)$ is by definition the subspace of $L^1(\Omega)$ consisting of those $\rho$'s in $L^1(\Omega)$ such that $J(\rho)$ is finite. The following fourth-order nonlinear evolution equation 
\begin{equation}\label{pde4}
\partial_t \rho+ \dive \Big(\rho \;  \nabla  \dive  \Big( \frac{\nabla \rho}{\vert \nabla \rho \vert} \Big)  \Big)=0, \mbox{ in $(0,T)\times\Omega$}, \; \rho_{\vert_{t=0}}=\rho_0,  
\end{equation}
supplemented by the zero-flux boundary condition
\begin{equation}\label{zerofluxbc}
\; \rho \nabla  \dive   \Big(  \frac{\nabla \rho}{\vert \nabla \rho \vert}   \Big) \cdot \nu=0 \mbox{ on $\partial \Omega$}
\end{equation}
 has been proposed in  \cite{burger2012regularized} for the purpose of denoising image densities. Numerical schemes for approximating the solutions of this equation have been investigated in \cite{burger2012regularized,during2012high,benning2013primal}. One should consider weak solutions and in particular interpret the nonlinear term $\dive(\frac{\nabla \rho}{\vert \nabla \rho\vert})$ as the negative of an element of the subdifferential of $J$ at $\rho$.

 At least formally,  when $\rho_0$ is a probability density on $\Omega$, \pref{pde4}-\pref{zerofluxbc} can be viewed as the Wasserstein gradient flow of $J$ (we refer to the textbooks of Ambrosio, Gigli, Savar\'e \cite{ambrosio2008gradient} and Santambrogio \cite{santambrogio2015optimal}, for a detailed exposition). Following the seminal work of Jordan, Kinderlehrer and Otto \cite{jordan1998variational} for the Fokker-Planck equation, it is reasonable to expect that solutions of (\ref{pde4}) can be obtained, at the limit $\tau\to 0^+$, of the JKO Euler implicit scheme:
\begin{equation}\label{jkotv0}
\rho_0^\tau=\rho_0, \; \rho_{k+1}^\tau \in \argmin  \Big\{ \frac{1}{2\tau} W_2^2(\rho_{k}^\tau, \rho)+ J(\rho), \; \rho \in \BV(\Omega)\cap \PP_2(\Omb)\Big\}
\end{equation}
where $\PP_2(\Omb)$ is the space of Borel probability measures $\Omb$ with finite second moment and $W_2$ is the quadratic Wasserstein distance:
\begin{equation}\label{defw2}
W_2^2(\rho_0, \rho_1):=\inf_{\gamma\in \Pi(\rho_0, \rho_1)}  \Big\{ \int_{\R^d\times \R^d} \vert x-y\vert^2 \mbox{d} \gamma(x,y)    \Big\},
\end{equation}
$\Pi(\rho_0, \rho_1)$ denoting the set of transport plans between $\rho_0$ and $\rho_1$ i.e. the set of probability measures on $\R^d\times \R^d$ having $\rho_0$ and $\rho_1$ as marginals. Our aim is to study in detail the discrete TV-JKO scheme \pref{jkotv0} as well as its connection with (suitable weak solutions) of the PDE \pref{pde4}.  Although the assertion  that \eqref{pde4} is the TV Wasserstein gradient flow is central to  the numerical schemes described in \cite{burger2012regularized,during2012high,benning2013primal}, there has been so far, to the best of our knowledge, no theoretical justification of this fact. 

\smallskip

Fourth-order equations which are Wasserstein gradient flows of functionals involving the gradient of $\rho$, such as the Dirichlet energy or the Fisher information, have been studied by McCann, Matthes and Savar\'e \cite{matthes2009fourth} who found a new method, \emph{the flow interchange technique}, to prove higher-order  estimates, we refer to  \cite{Matthes2016} for a recent reference on this topic. The total variation is however too singular for such arguments to be directly applicable, as far as we know. We shall prove the convergence of JKO steps as $\tau \to 0^+$ under the extra assumption that densities remain bounded aways from zero. Whether this  extra assumption is reasonable or not is related to a minimum principle issue, interesting in its own right, namely the monotonicty of the infimum along JKO steps. We shall see that, in a convex domain, JKO steps obey a maximum principle (the maximum of the density is nonincreasing along JKO steps). The corresponding minimum principle seems more difficult to prove and we have been able to establish it only in some particular cases, namely in dimension one and in the radially symmetric case, eventhough we conjecture it is satisfied in more general situations.

\smallskip

The paper is organized as follows. In section \ref{sec-examples}, we start with the discussion of a few examples. Section \ref{sec-opti} establishes optimality conditions for JKO steps thanks to an  entropic regularization scheme. Section \ref{sec-max} is devoted to some properties of solutions of JKO steps and in particular a maximum principle based on a result of \cite{de2016bv}, we also establish a minimum principle in dimension one and in the radially symmetric case.  Finally,  in section \ref{sec-conv}, we prove  a conditional convergence result, we establish convergence of the JKO scheme, as $\tau\to 0^+$, under the extra assumption that the density remains away from zero, this covers the unidimensional case as well as the  radially symmetric case when the initial conditon is strictly positive.

\section{Some examples}\label{sec-examples}

We first  recall the Kantorovich dual formulation of $W_2^2$:
\begin{equation}\label{dualformW2}
\frac{1}{2}W_2^2(\mu_0, \mu_1)=\sup \Big\{\int_{\R^d} \psi \mbox{d} \mu_0 + \int_{\R^d} \varphi \mbox{d}\mu_1 \; : \; \psi(x)+\varphi(y)\le \frac{\vert x-y\vert^2}{2}  \Big\}
\end{equation}
an optimal pair $(\psi, \varphi)$ for this problem is called a pair of Kantorovich potentials. The existence of Kantorovich potentials is well-known and such potentials can be taken to be conjugates of each other, i.e. such that
\[\varphi(x)=\inf_{y\in \R^d} \{\frac{1}{2}\vert x-y\vert^2-\psi(y)\}, \; \psi(y)=\inf_{x\in \R^d} \{\frac{1}{2}\vert x-y\vert^2-\varphi(x)\},\]
which implies that $\varphi$ and $\psi$ are semi-concave (more precisely $\frac{1}{2} \vert . \vert^2-\varphi$ is convex). If $\mu_1$ is absolutely continuous with respect to the $d$-dimensional Lebesgue measure, $\varphi$ is differentiable $\mu_1$ a.e. and the map $T=\id - \nabla \varphi$ is the gradient of a convex function pushing forward $\mu_1$ to $\mu_0$ which is in fact the optimal transport between $\mu_0$ and $\mu_1$ thanks to Brenier's theorem \cite{brenier91}. In such a case, we will simply refer to $\varphi$ as a Kantorovich potential between $\mu_1$ and $\mu_0$. We refer the reader to \cite{villanitop} and \cite{santambrogio2015optimal} for details.

\smallskip

In this section, we will consider some explicit examples which rely on the following sufficient optimality condition (details for a rigorous derivation of the Euler-Lagrange equation for JKO steps will be given  in section \ref{sec-opti}) in the case of the whole space i.e. $\Omega=\R^d$. Let us also recall that by Sobolev inequality $\BV(\R^d)$ is continuously embedded in $L^{\frac{d}{d-1}}(\R^d)$.

\begin{lem}\label{1dsufficient}
Let $\rho_0 \in \PP_2(\R^d)$, $\tau>0$ and $\Omega=\R^d$ (so $J$ is the total variaton on the whole space), if $\rho_1\in \BV(\R^d)\cap \PP_2(\R^d)$ is such that 
\begin{equation}\label{condsuff}
\frac{\varphi}{\tau} + \dive(z)\ge 0, \mbox{ with equality $\rho_1$-a.e.} 
\end{equation}
where $\varphi$ is a Kantorovich potential between $\rho_1$ and $\rho_0$ and $z\in C^1(\R^d)$, with  $\Vert z \Vert_{L^\infty} \le 1$, $\dive(z)\in L^d(\R^d)$ (so that $\dive(z) \rho_1\in L^1(\R^d)$), and 
\begin{equation}\label{zsubdiff}
J(\rho_1)=\int_{\R^d} \dive(z) \rho_1.
\end{equation}
Then, setting 
\begin{equation}\label{defPhi}
\Phi_{\tau, \rho_0}(\rho):=\frac{1}{2\tau} W_2^2(\rho_0, \rho)+J(\rho), \; \forall \rho \in \BV(\R^d)\cap \PP_2(\R^d)
\end{equation}
one has 
\[\Phi_{\tau, \rho_0}(\rho_1) \le \Phi_{\tau, \rho_0}(\rho), \; \forall \rho \in \BV(\R^d)\cap \PP_2(\R^d).\]
\end{lem}

\begin{proof}
For all $\rho\in \BV(\R^d)\cap \PP_2(\R^d)$, $J(\rho)\ge \int_{\R^d} \dive(z) \rho=J(\rho_1)+\int_{\R^d} \dive(z) (\rho-\rho_1)$, and it follows from  the Kantorovich duality formula that
\[\frac{1}{2\tau} W_2^2(\rho_0, \rho)\ge \frac{1}{2\tau} W_2^2(\rho_0, \rho_1) +\int_{\R^d} \frac{\varphi}{\tau} (\rho-\rho_1).\]
The claim then directly follows from \pref{condsuff}. 
\end{proof}


\subsection{The case of a characteristic function}\label{subsec-charac}

A simple  illustration of  Lemma \ref{1dsufficient} in dimension 1 concerns the case of a uniform $\rho_0$, (here and in the sequel we shall denote by $\chi_A$ the characteristic function of the set $A$):
\[\rho_{0}=\rho_{\alpha_0}, \; \alpha_0 >0, \; \rho_{\alpha}:=\frac{1}{2\alpha} \chi_{[-\alpha, \alpha]}.\]
It is natural to make the ansatz that the minimizer of $\Phi_{\tau, \rho_0}$ defined by \pref{defPhi} remains of the form $\rho_1=\rho_{\alpha_1}$ for some $\alpha_1>\alpha_0$. The optimal transport between $\rho_{\alpha_1}$ and $\rho_0$ being the linear map $T=\frac{\alpha_0}{\alpha_1} \id$, a direct computation gives
\[\Phi_{\tau, \rho_0}(\rho_{\alpha_1})=\frac{1}{\alpha_1} +\frac{1}{6\tau}(\alpha_1-\alpha_0)^2\]
which is minimal when $\alpha_1$ is the only root in $(\alpha_0, +\infty)$ of 
\begin{equation}\label{optial}
\alpha_1^2(\alpha_1-\alpha_0)=3\tau.
\end{equation}
To check that this is the  correct guess, we shall check that the conditions of Lemma \ref{1dsufficient} are met. It is easy to check that the potential defined by
\[\varphi(x)=\frac{1}{2\alpha_1} (\alpha_1-\alpha_0) x^2-\frac{3\tau }{2 \alpha_1}\]
is a Kantorovich potential between $\rho_1=\rho_{\alpha_1}$ and $\rho_0$.  Define\footnote{The guess for this construction is by integrating the Euler-Lagrange equation on the support of $\rho_{\alpha_1}$.} then $z_1$ by 
\[\tau z_1(x): =-\frac{(\alpha_1-\alpha_0)}{6\alpha_1} x^3+\frac{3 \tau x} {2 \alpha_1} , \; x\in [-\alpha_1, \alpha_1]\]
extended by $1$ on $[\alpha_1, +\infty)$ and $-1$ on $(-\infty, -\alpha_1]$.  By construction $-1\le z_1\le 1$ (use the fact that it is odd and nondecreasing on $[0, \alpha_1]$ thanks to \pref{optial}), also $z_1'(\pm \alpha_1)=0$ so that $z_1 \in C^1(\R)$ and $z_1(\alpha_1)=1$, $z_1(-\alpha_1)=-1$ and one easily checks that $J(\rho_1)=-\int_{\R} z_1  D \rho_1=\int_{\R} z'_1 \rho_1$ (here and in the sequel $D\rho_1$ denotes the Radon measure which is the  distributional derivative of the $\BV$ function $\rho_1$). Moreover $\tau z_1'+\varphi\ge 0$ with an equality on $[-\alpha_1, \alpha_1]$. The optimality of $\rho_1=\rho_{\alpha_1}$ then directly follows from  Lemma \ref{1dsufficient}.

\smallskip

Of course, the argument can be iterated so as  to obtain the full TV-JKO sequence:
\[\rho_{k+1}^\tau=\argmin \;  \Phi_{\tau, \rho_k^\tau} =\Big(\frac{\alpha_{k+1}^\tau}{ \alpha_k^\tau } \id\Big)_\# \rho_k^\tau= \Big(\frac{\alpha_{k+1}^\tau}{ \alpha_0} \id\Big)_\# \rho_0\]
where $\alpha_k^\tau$ is defined inductively by
\[(\alpha_{k+1}^\tau-\alpha_k^\tau)(\alpha_{k+1}^\tau)^2=3 \tau, \; \alpha_0^\tau=\alpha_0\]
which is nothing but the implicit Euler discretization of the ODE
\[\alpha' \alpha^2=3, \; \alpha(0)=\alpha_0,\]
whose solution is $\alpha(t)=(\alpha_0^3+ 9t)^{\frac{1}{3}}$. Extending $\rho_{k}^\tau$ in a piecewise constant way: $\rho^\tau(t)=\rho_{k+1}^\tau$ for $t \in (k\tau, (k+1)\tau]$, it is not difficult to check that $\rho^\tau$ converges (in $L^{\infty}((0,T), (\PP_2(\R), W_2))$ and in $L^p((0,T)\times \R)$ for any $p\in (1, \infty)$ and any $T>0$) to $\rho$ given by $\rho(t,.)=(\frac{\alpha(t)}{\alpha_0} \rm{id})_\#\rho_0$. Since $v(t,x)=\frac{\alpha'(t)}{\alpha(t)} x$ is the velocity field associated to $X(t,x)=\frac{\alpha(t)}{\alpha_0}x$, $\rho$ solves the continuity equation
\[\partial_t \rho +(\rho v)_x=0.\]
In addition, $\rho v=-\rho z_{xx}$ where
\[ z(t,x)=\frac{-\alpha'(t)}{6 \alpha(t)} x^3+ \frac{ 3 x}{2\alpha(t)}, \; x\in[-\alpha(t), \alpha(t)],\]
extended by $1$ (respectively $-1$) on $[\alpha(t), +\infty)$ (respectively $(-\infty, -\alpha(t)]$). The function $z$  is $C^1$, $\Vert z \Vert_{L^\infty} \le 1$ and $z \cdot D \rho =-\vert D \rho\vert$ (in the sense of measures). In other words the limit $\rho$ of $\rho^\tau$ satisfies
\[\partial_t \rho-(\rho z_{xx})_x=0\]
with $\vert z \vert \le 1$ and $z \cdot D\rho =-\vert  D\rho\vert$ which is the natural weak form of \pref{pde4} since $z_{xx}=\nabla \dive(z)$  in dimension one.  


\subsection{Instantaneaous creation of discontinuities}\label{subsec-discont}

We now consider the case where $\rho_0(x)=(1-\vert x\vert)_+$ and will show that the JKO scheme instantaneously creates a discontinuity at the level of $\rho_1$, the minimizer of $\Phi_{\tau, \rho_0}$  when $\tau$ is small enough. We indeed look for $\rho_1$ in the form:
\[\rho_1(x) = \begin{cases}
1-\beta/2  &\mbox{ if $\vert x\vert <   \beta$}, \\
(1-\vert x \vert)_+  &\mbox{ if $\vert x\vert \ge   \beta$,}  \end{cases}\]
for some well-chosen $\beta\in (0,1)$. The optimal transport map $T$ between such a $\rho_1$ and $\rho_0$ is odd and given explicitly by

\[T(x)=\begin{cases}
1-\sqrt{1-x(2-\beta)}  &\mbox{ if $x \in [0,\beta)$},\\
x  & \mbox{ if $x\ge \beta$}. \end{cases}\]

The Kantorovich potential which vanishes at $\beta$ (extended in an even way to $\R_-$) is then given by
\[\varphi(x) = \begin{cases}
\frac{x^2}{2}- x -\frac{(1-x(2-\beta))^{3/2}}{3(1-\beta/2)} +C &\mbox{ if $x \in [0,\beta)$},\\
0 & \mbox{ if $x>\beta$},
\end{cases}\]
where
\[C = -\frac{\beta^2}{2} + \beta +\frac{2(1-\beta)^{3}}{3(2-\beta)}. \]
Let  us now integrate $\tau z'=-\varphi$ on $[0, \beta]$ with initial condition $z(0)=0$, i.e. for $x\in [0, \beta]$
\[\begin{split}
\tau z(x)=&-\frac{x^3}{6} +\frac{x^2}{2}-\frac{4}{15(2-\beta)^2} [1-(1-2\beta)x]^{\frac{5}{2}}\\
&+\Big(\frac{\beta^2}{2} - \beta -\frac{2(1-\beta)^{3}}{3(2-\beta)}\Big) x + \frac{4}{15(2-\beta)^2}
\end{split} \]
Note that $z$ is nondecreasing on $[0, \beta]$  (because  $\varphi(0)<0$, $\varphi(\beta)=0$ and $\varphi$ is convex on $[0, \beta]$ so that $\varphi\le 0$ on $[0, \beta]$), our aim now is to find $\beta\in (0,1)$ in such a way that $z(\beta)=1$ i.e. replacing in the previous formula 
\[\tau=\frac{\beta^3}{3} -\frac{\beta^2}{2}+ \frac{4(1-(1-\beta)^5)}{15(2-\beta)^2}-\frac{2(1-\beta)^3 \beta}{3(2-\beta)}\]
the right hand-side is a continuous function of $\beta\in [0,1]$ taking value $0$ for $\beta=0$ and $\frac{1}{10}$ for $\beta=1$, hence as soon as $10\tau <1$ one may find a $\beta\in (0,1)$ such that indeed $z(\beta)=1$. Extend then $z$ by $1$ on $[\beta, +\infty)$ and to $\R_-$ in an odd way. We then have built a function $z$ which is $C^1$ ($\varphi(\beta)=0$), such that $\vert z \vert \le 1$, $z\cdot D \rho_1=- \vert  D\rho_1 \vert$ and such that $z'+\frac{\varphi}{\tau}=0$. Thanks to Lemma \ref{1dsufficient}, we conclude that $\rho_1$ is optimal. This example shows that discontinuities may appear at the very first iteration of the TV-JKO scheme.

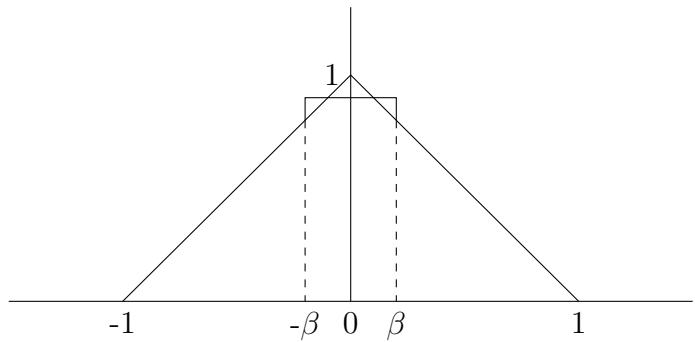
\begin{figure}\label{hatrho}
\begin{center}
\begin{tikzpicture}[scale=3]

\draw (1,0) -- (0,1) -- (-1,0);
\draw (-.2,.8)--(-.2,.9)--(.2,.9)--(.2,0.8);
\draw[dashed] (-.2,.8) -- (-.2,0);
\draw[dashed] (.2,.8) -- (.2,0);
\draw (1.5,0) -- (1,0)node[below] {1} --(.2,0) node[below] {$\beta$}  -- (0,0) node[below] {0} --(-.2,0) node[below] {-$\beta$} -- (-1,0) node[below] {-1}--(-1.5,0);

\draw (0,0) -- (0,1)node[left]{1}-- (0,1.3);

\end{tikzpicture}
\end{center}
\caption{The probablity density functions $\rho_0$ and $\rho_1$ from section \ref{subsec-discont}}
\end{figure}

\section{Euler-Lagrange equation for JKO steps }\label{sec-opti}

 The aim of this section is to establish optimality conditions for \pref{onestepjko}. Despite the fact that it is a convex minimization problem, it involves two nonsmooth terms $J$ and $W^2_2(\rho_0,.)$, so some care should be taken of to justify rigorously the arguments. In the next subsection, we introduce an entropic regularization, the advantage of this strategy is that the minimizer will be positive everywhere, giving some differentiability of the transport term.  
 
 \subsection{Entropic approximation}\label{subsec-entreg}
 
 In this whole section, we assume that $\Omega$ is an open bounded connected (not necessarily convex) subset of $\R^d$ with  Lipschitz boundary and denote by $\PPa$ the set of Borel probability measures on $\Omega$ that are absolutely continuous with respect to the Lebesgue measure (and will use the same notation for $\mu\in \PPa$ both for the measure $\mu$ and its density). Given $\rho_0 \in \PPa$ and $\tau>0$, we consider one step of the TV-JKO scheme:
\begin{equation}\label{onestepjko}
\inf_{\rho \in \PPa}  \Big\{ \frac{1}{2\tau} W_2^2(\rho_0, \rho)+ J(\rho) \Big\}.
\end{equation}
It is easy by the direct method of the calculus of variations to see that \pref{onestepjko} has at least one solution, moreover $J$ being convex and $\rho \mapsto W_2^2(\rho, \rho_{0})$ being strictly convex whenever $\rho_0 \in \PPa$ (see \cite{santambrogio2015optimal}), the minimizer is in fact unique, and in the sequel we denote it by  $\rho_1$.  Given $h>0$ we consider the following approximation of \pref{onestepjko}:
 \begin{equation}\label{onestepentr}
 \inf_{\rho\in \PPa} \Big\{   \F_h(\rho):= \frac{1}{2\tau} W_2^2(\rho_0, \rho)+J(\rho)+h \EE(\rho) \Big\} 
 \end{equation}
 where
 \[\EE(\rho):=\int_{\Omega} \rho(x) \log(\rho(x)) \mbox{d}x.\]
It is easy to  to see that \pref{onestepentr} admits a unique solution $\rho_h$. Moreover, since $\Omega$ is bounded, $\EE$ is lower bounded, hence  $J(\rho_h)$ is bounded. Recalling  that the embedding $BV(\Omega)\subset L^p(\Omega)$ is compact for every  $p\in[1, \frac{d}{d-1})$, one may therefore (up to extraction) assume that $\rho_h$ converges as $h\to 0$ a.e. and strongly in $L^p(\Omega)$ for every $p\in[1, \frac{d}{d-1})$ to some $\rho_1$, which, by a standard $\Gamma$-convergence argument, is easily seen to be the solution of \pref{onestepjko}. The advantage of this regularization is that not only each $\rho_h$ is bounded from below but also that $h \log(\rho_h)$ is bounded from below uniformly in $h$ (but not in $\tau$ which is fixed throughout this section):



\begin{prop}\label{boundbelowh}
Up to passing to a subsequence, the family $\beta_h :=h\log(\rho_h)$ is uniformly bounded from below. Moreover, $\beta_h$ is  bounded in $L^{p}(\Omega)$ for any $p>1$ and $\max(0, \beta_h)$ converges strongly to $0$  in $L^{p}(\Omega)$ for any $p>1$.

\end{prop}

\begin{proof}
Let $t_h>0$ be such that the set $F_{t_h}^h := \{\rho_h >t_h\}$ has positive measure and finite perimeter (recall that $\rho_h\in \BV$). Let us assume that there is  an $\eps\in (0,1)$ such that 
\begin{equation}\label{smalleps}
 \eps  \le \frac{t_h \vert F_{t_h}^h \vert }{2\vert \Omega\vert}, 
\end{equation}
and
\begin{equation}
\vert A_{\eps, h}\vert >0 \mbox{ with } A_{\eps,h}:=\{\rho_h \le \eps\}.
\end{equation}
We aim to show that $\eps$ cannot be arbitrarily small.
Define then $\mu_{\eps, h}:=\max( \rho_h, \eps)$ that is $\eps$ on $A_{\eps, h}$ and $\rho_h$ elsewhere. Defining $c_{\eps,h}:=\int_{\Omega} (\mu_{\eps,h}- \rho_h)$ and observing that $c_{\eps,h} \le \eps \vert \Omega\vert $, we see that  \pref{smalleps} implies that $c_{\eps,h}\le \frac{1}{2} t_h \vert F_{t_h}^h \vert$ and $t_h \ge 2 \eps$ so that $A_{\eps,h}$ and $F_{t_h}^h$ are disjoint. Finally, set
\begin{equation}\label{defrhoepsh}
\rho_{\eps,h}:=\mu_{\eps,h}-c_{\eps,h} \frac{  \chi_{F_{t_h}^h}}{ \vert F_{t_h}^h\vert}.
\end{equation}
See Figure \ref{figrhoepsh}, where we set $\tilde c_{\eps,h} :=  c_{\eps,h}/ \vert F_{t_h}^h\vert$.

\begin{figure}
\includegraphics[width = \textwidth]{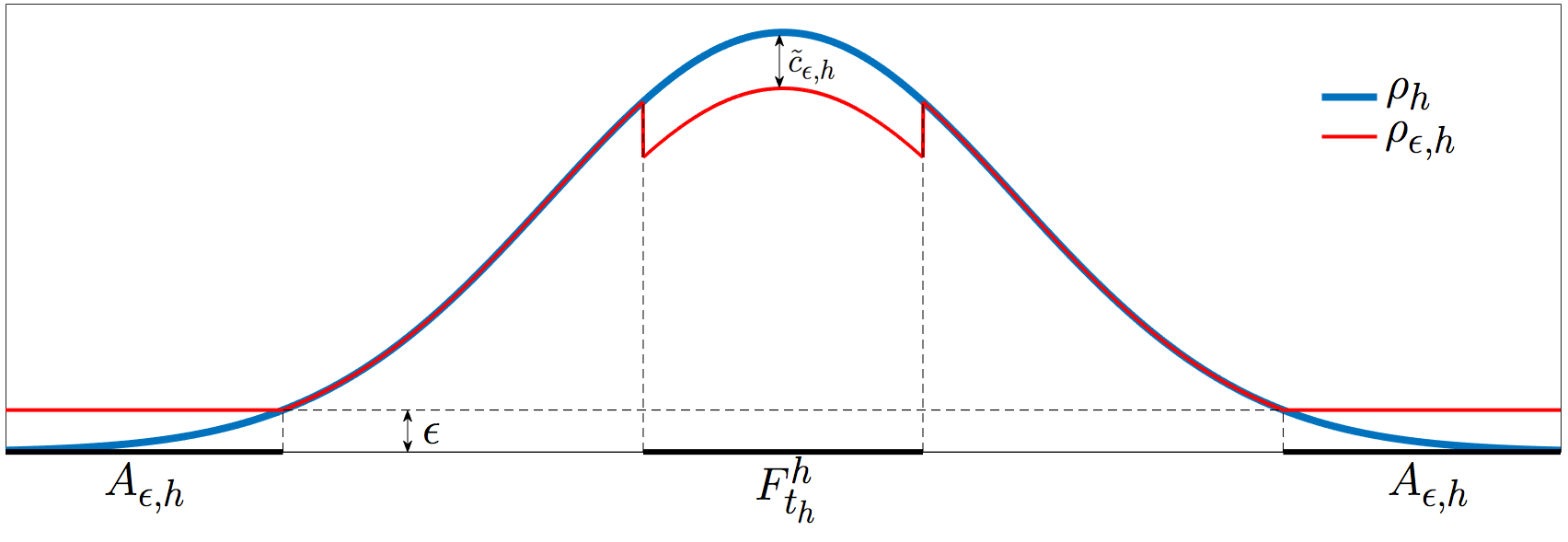}
\caption{The perturbation $\rho_{\eps, h}$ from  \pref{defrhoepsh}. \label{figrhoepsh}}
\end{figure}

\smallskip

By construction $\rho_{\eps,h} \in \PP(\Omega)$ hence $0\le \F_h(\rho_{\eps,h})-\F_h(\rho_h)$, in this difference we have four terms, namely 
\begin{itemize}

\item the Wasserstein term, which, using the Kantorovich duality formula \pref{dualformW2} and the fact that $\Omega$ is bounded can be estimated in terms of $\Vert \rho_{\eps,h}-\rho_h\Vert_{L^1} =2 c_{\eps,h}$:
\begin{equation}\label{estim1}
\frac{1}{2\tau} W_2^2(\rho_{\eps,h}, \rho_0)-\frac{1}{2\tau} W_2^2(\rho_h, \rho_0) \le \frac{C}{\tau}   c_{\eps,h}.
\end{equation}
for a constant $C$ that depends on $\Omega$ but neither on $\eps$ nor $h$,

\item the TV term: $J(\rho_{\eps,h})-J(\rho_h)$: outside $F_{t_h}^h$ we have replaced $\rho_h$ by a $1$-Lipschitz function of $\rho_h$ which decreases the TV semi-norm, on $F_{t_h}^h$ on the contrary we have created a jump of magnitude $c_{\eps, h} /\vert F_{t_h}^h\vert$ so
 \begin{equation}\label{estim2}
J(\rho_{\eps,h})-J(\rho_h) \le c_{\eps,h} \frac{\Per(F_{t_h}^h)}{\vert F_{t_h}^h\vert} 
\end{equation}
where $\Per(F_{t_h}^h)=J(\chi_{F_{t_h}^h})$ denotes the perimeter of $F_{t_h}^h$ (in $\Omega$),

\item the entropy variation on $A_{\eps,h}$, on this set both $\rho_{\eps,h}$ and $\rho_h$ are less than $\eps$ so that $(1+\log(t)) \le (1+\log(\eps))$ whenever $t\in [\rho_h, \rho_{\eps,h}]$ which by the mean value theorem
 yields 
 \begin{equation}\label{estim3}
\int_{A_{\eps,h}} (\rho_{\eps, h} \log(\rho_{\eps, h})     -\rho_h \log (\rho_h)) \le (1+\log(\eps)) c_{\eps,h}
\end{equation}

\item the last term is the entropy variation on $F_{t_h}^h$. It is convenient to split $F_{t_h}^h$ into  $F_{t_h}^h\cap \{\rho_{\eps, h} \ge \frac{1}{e}\}$ and $F_{t_h}^h\cap \{\rho_{\eps, h} < \frac{1}{e}\}$. The entropy variation on the first part is easy to control. Indeed, $t\mapsto t\log(t)$ is nondecreasing on $[\frac{1}{e}, +\infty)$. Since, on $F_{t_h}^h\cap \{\rho_{\eps, h} \ge \frac{1}{e}\}$,  $\rho_{h} \ge \rho_{\eps, h} \ge \frac{1}{e}$, we have $(\rho_{\eps, h} \log(\rho_{\eps,h})-\rho_h\log (\rho_h)) \le 0$. As for the second part,  we observe that $F_{t_h}^h \cap \{\rho_{\eps, h} <  \frac{1}{e}\} \subset \{ \rho_h \le \frac{1}{e}+\frac{t_h}{2}\}$,  so on this set, both $\rho_{\eps, h}$ and $\rho_h$ remain in the interval $[\frac{t_h}{2},  \frac{1}{e} +\frac{t_h}{2}]$. We thus have
\begin{equation}\label{estim4}
\int_{F_{t_h}^h}  (\rho_{\eps, h} \log(\rho_{\eps, h})     -\rho_h \log (\rho_h)) \le  C_h(t_h) c_{\eps,h},
\end{equation}
where 
\begin{equation}\label{defchmh}
C_h(t_h):=\max  \Big\{ \vert 1+ \log(t) \vert \;   : \; \frac{t_h}{2}   \leq t \leq   \frac{1}{e} +\frac{t_h}{2}  \Big\}.
\end{equation}

\end{itemize}

Putting together (\ref{estim1})-(\ref{estim2})-(\ref{estim3})-(\ref{estim4}), we  arrive at
\[0\le \Big(\frac{C}{\tau}+  \frac{\Per(F_{t_h}^h)}{\vert F_{t_h}^h\vert} +  h C_h(t_h)+ h   \log( \eps) +h \Big) c_{\eps,h}\]
which for small enough $\eps$ is possible only when $c_{\eps, h}=0$ i.e. $\vert A_{\eps,h}\vert=0$.  More precisely, either we have the lower bound:
\begin{equation}
\label{boundbelowteta}
 h \log(\rho_h) \ge -\frac{C}{\tau}-h C_h(t_h)- \frac{\Per(F_{t_h}^h)}{\vert F_{t_h}^h\vert} -h
\end{equation}
or \pref{smalleps} is impossible i.e.  $\rho_h \ge \frac{t_h \vert F_{t_h}^h \vert }{2\vert \Omega\vert}$. To prove that $\beta_h=h \log(\rho_h)$ is bounded from below uniformly in $h$, it is therefore enough to show that we can find a family $t_h$, bounded and bounded away from $0$,  such that $|F_{t_h}^h|$ remains bounded away from $0$, and $\mathrm{Per}(F_{t_h}^h)$ is uniformly bounded from above as $h\to 0$. First note that, since $J(\rho_h)$ is bounded,  there exists $\rho$ such that $\rho_h \to \rho$ in $L^1$ and a.e. up to a subsequence, note also that $\rho\in  \BV$ and $\rho$ is a probability density. Setting  $F_t:=\{\rho >t\}$, $F_t^h:=\{\rho_h >t\}$, if $s>t$, since $\rho_h$ converges a.e. to $\rho$, we have a.e.  $\liminf_h \chi_{F_t^h} \ge \chi_{F_s}$. It then follows from Fatou's Lemma that when $s>t$, $\liminf_h \vert F_t^h \vert \ge \vert F_s\vert$, hence  choosing    $0<\beta_1 <\beta_2<\beta$ so that $\vert F_{\beta} \vert >0$, we deduce that there exists $h_0>0$ and  $c_1>0$ such that for all $t\in [\beta_1,\beta_2]$ and all $h\in (0, h_0]$, we have  $c_1 \leq |F_t^{h} | \leq \vert \Omega\vert$. For an upper bound on perimeters, we observe that since  $J(\rho_h) \leq C$, thanks to the co-area formula, we have
$$
\int_{\beta_1}^{\beta_2}\Per(F_t^h) \mathrm{d}t\leq J(\rho_h)\leq C.
$$
So, there exists $t_h\in [\beta_1,\beta_2]$ such that
$
\Per(F_{t_h}^h)\leq C/(\beta_2-\beta_1).
$

\smallskip

Finally, since $\rho_h$ converges in $L^1$,  we may assume that, up to a subsequence, $\rho_h \le \phi$ for some $\phi \in L^1$ (see Theorem IV.9 in \cite{Brezis}). Then, by Dominated convergence and since $\log(\max(\phi,1))\in L^p(\Omega)$ for every $p>1$, we have that $\log(\max(\rho_h,1))$ converges a.e. and in $L^p$, in particular this implies that $\max(0, \beta_h)$ converges to $0$ strongly in $L^p(\Omega)$, and we have just seen that $\min(0, \beta_h)$ is bounded in $L^\infty( \Omega)$.

\end{proof}

Let us also recall some well-known facts (see \cite{chambolleintro})  about the total variation functional $J$ viewed as a convex l.s.c. and one-homogeneous functional on  $L^{\frac{d}{d-1}}(\Omega)$. Define 
\begin{equation}
\Gamma_d:=\Big\{ \xi\in L^d(\Omega) \; : \; \exists z\in L^{\infty} (\Omega, \R^d), \; \Vert z \Vert_{L^{\infty}} \le 1, \; \dive(z)=\xi, \; \; z\cdot \nu =0 \mbox{ on $\partial \Omega$}\Big\}
\end{equation}
where $\dive(z)=\xi, \;  z\cdot \nu =0$ on $\partial \Omega$ are to be understood in the weak sense
\[\int_{\Omega} \xi u=-\int_{\Omega} z \cdot \nabla u, \; \forall u\in C^1(\Omb).\]
Note that $\Gamma_d$ is closed and convex in $L^d(\Omega)$ and $J$ is its support function:
\begin{equation}\label{tvsupportf}
J(\mu)=\sup_{\xi \in \Gamma_d} \int_{\Omega} \xi \mu, \; \forall \mu \in L^{\frac{d}{d-1}}(\Omega).
\end{equation}
As for the Wasserstein term, recalling Kantorovich dual formulation \pref{dualformW2},  the derivative of the Wasserstein term $\rho \mapsto W^2_2(\rho_0, \rho)$ term will be expressed in terms of a Kantorovich potential between $\rho$ and $\rho_0$.

\smallskip

We then have the following characterization for $\rho_h$:

\begin{prop}\label{optih}
There exists $z_h \in L^{\infty}(\Omega, \R^d)$ such that $\dive(z_h)\in L^{p}(\Omega)$ for every $p\in [1,+\infty)$, $\Vert z_h \Vert_{L^{\infty}} \le 1$, $z_h \cdot \nu=0$ on $\partial \Omega$, $J(\rho_h)=\int_{\Omega} \dive(z_h) \rho_h$ and
\begin{equation}\label{och}
\frac{\varphi_h}{\tau}+ \dive(z_h)+h \log(\rho_h)=0, \mbox{ a.e. in $\Omega$} 
\end{equation}
where $\varphi_h$ is the Kantorovich potential between $\rho_h$ and $\rho_0$.
\end{prop}

\begin{proof} 
 Let $\mu \in L^{\infty}(\Omega)\cap \BV(\Omega)$ such that $\int_{\Omega} \mu=0$. Thanks to Proposition \ref{boundbelowh}, we know that $\rho_h$ is bounded away from $0$ hence for small enough $t>0$, $\rho_h + t \mu$ is positive hence a probability density. Also, as a consequence of Theorem 1.52 in \cite{santambrogio2015optimal}, we have that 
\begin{equation}\label{derivW}
\lim_{t\to 0^+} \frac{1}{2t}[W_2^2(\rho_0, \rho_h+t \mu)-W_2^2(\rho_0, \rho_h)]=\int_{\Omega} \varphi_h   \mu
\end{equation}
where $\varphi_h$ is the (unique up to an additive constant) Kantorovich potential between $\rho_h$ and $\rho_0$, in particular $\varphi_h$  is Lipschitz and semi concave ($D^2 \varphi_h \le \id$ in the sense of measures and $\id-\nabla \varphi_h$ is the optimal transport between $\rho_h$ and $\rho_1$). By the optimality of $\rho_h$ and the fact that $J$ is a semi-norm, we get
\begin{equation}\label{varineq}
J(\mu) \ge J(\rho_h+\mu)-J(\rho_h)\ge \lim_{t\to 0^+} t^{-1} (J(\rho_h+t \mu)-J(\rho_h))    \ge \int_{\Omega} \xi_h \mu,
\end{equation}
 where 
 \[ \xi_h:=-\frac{\varphi_h}{\tau} -h \log(\rho_h).\]
Since $\varphi_h$ is defined up to an additive constant, we may chose it in such a way that $\xi_h$ has zero mean, doing so, \pref{varineq} holds for any $\mu \in L^{\infty}(\Omega)\cap \BV(\Omega)$ (not necessarily with zero mean). Being Lipschitz, $\varphi_h$ is bounded,  also observe that $h(\log(\rho_h))_+=h\log(\max(1, \rho_h))$ is in $L^p(\Omega)$ for every $p\in [1,+\infty)$ since $\rho_h\in L^{\frac{d}{d-1}}(\Omega)$ and $h \log(\rho_h)_-=-h\log(\min(1, \rho_h))$ is $L^{\infty}(\Omega)$ thanks to Proposition \ref{boundbelowh}, hence we have $\xi_h \in L^{p}(\Omega)$ for every $p\in [1,+\infty)$.

 By approximation and observing that $\xi_h \in L^{d}(\Omega)$, \pref{varineq} extends to all $\mu\in L^{\frac{d}{d-1}}(\Omega)$. In particular, we have
\[\sup_{\xi \in \Gamma_d} \int_{\Omega} \xi \mu \ge \int_{\Omega} \xi_h \mu\]
but since $\Gamma_d$ is convex and closed in $L^d(\Omega)$, it follows from Hahn-Banach's separation theorem that $\xi_h \in \Gamma_d$. Finally, getting back to \pref{varineq} (without the zero mean restriction on $\mu$) and taking $\mu=-\rho_h$ gives $J(\rho_h)\leq \int_{\Omega} \xi_h \rho_h$, and we then deduce  that this should be an equality.

\smallskip
 
\end{proof}

 \subsection{Euler-Lagrange equation}\label{subsec-oc}

We are now in position to rigorously establish the Euler-Lagrange equation for \pref{onestepjko}:

\begin{thm}\label{eulertvjko}
If $\rho_1$ solves \pref{onestepjko}, there exists $\varphi$ a Kantorovich potential between $\rho_1$ and $\rho_0$ (in particular $\id -\nabla \varphi$ is the optimal transport between $\rho_1$ and $\rho_0$), $\beta\in L^{\infty} (\Omega)$, $\beta\ge 0$ and $z\in L^{\infty}(\Omega, \R^d)$ such that
\begin{equation}\label{eulerjko}
\frac{\varphi}{\tau}+\dive(z)=\beta, \; z\cdot{\nu} =0 \mbox{ on $\partial \Omega$}, 
\end{equation}
and 
\begin{equation}\label{slackjko}
\beta \rho_1=0, \; \; \Vert z \Vert_{L^{\infty}} \le 1, \; J(\rho_1)=\int_{\Omega} \dive(z) \rho_1. 
\end{equation}

\end{thm}

\begin{rem}
It is not difficult (since \pref{onestepjko} is a convex problem) to check that \pref{eulerjko}-\pref{slackjko} are also sufficient optimality conditions. The main point here is that the right hand side $\beta$ in \pref{eulerjko}  which is a multiplier associated with the nonnegativity constraint is better than a measure, it is actually an $L^{\infty}$ function.
\end{rem}

\begin{proof}
As in section \ref{subsec-entreg}, we denote by $\rho_h$ the solution of the entropic approximation \pref{onestepentr}. Up to passing to a subsequence (not explicitly written),  we may assume that $\rho_h$ converges a.e. and strongly in $L^p(\Omega)$ (for any $p\in [1, \frac{d}{d-1})$) to $\rho_1$ (the solution of \pref{onestepjko}, again by a standard $\Gamma$-convergence argument). We then rewrite the Euler-Lagrange equation from Proposition \ref{optih} as
\begin{equation}
\frac{\varphi_h}{\tau}+ \dive(z_h)+ \beta_h^+=\beta_h^-, 
\end{equation}
where $\beta_h^+:=h\log(\max(\rho_h, 1))$, $\beta_h^-:=-h\log(\min(\rho_h, 1))$,
and 
\begin{equation}
\Vert z_h \Vert_{L^{\infty}}\le 1, \; z_h \cdot \nu=0 \mbox{ on $\partial \Omega$ and } J(\rho_h)=\int_{\Omega} \dive(z_h) \rho_h.
\end{equation}
It follows from Proposition \ref{boundbelowh}  that $\beta_h^+$ converges to $0$ strongly in any $L^p$, $p\in [1, +\infty)$ and  that $\beta_h^-$ is bounded in $L^{\infty}$. Up to subsequences, we may therefore assume that $z_h$ and $\beta_h^-$ weakly-$*$ converge in $L^{\infty}$ respectively to some  $z$ and $\beta$ with $\Vert z \Vert_{L^{\infty}}\le 1$, $z\cdot \nu=0$ on $\partial \Omega$ and $\beta \ge 0$. As for the Kantorovich potentials $\varphi_h$,  since the transport map $(\id-\nabla \varphi_h)$ a.e. takes values in $\Omega$ we have $\Vert \nabla \varphi_h \Vert_{L^{\infty}} \le {\rm{diam}}(\Omega)$, hence  $\varphi_h$ is an equi-Lipschitz family because $\Omega$ is bounded. Moreover $\int_{\Omega} \varphi_h=\tau \int_{\Omega} (\beta_h^- -\beta_h^+)$ which remains bounded, hence we may assume that $\varphi_h$ converges uniformly to some potential $\varphi$ and it is well-known (see \cite{santambrogio2015optimal}) that $\varphi$ is a Kantorovich potential between $\rho_1$ and $\rho_0$. Letting $h$ tend to $0$ gives \pref{eulerjko}.

\smallskip

Since $\rho_h$ converges strongly in $L^1$ to $\rho_1$ and $\beta_h^{-}$ converges weakly-$*$ to $\beta$ in $L^{\infty}$ we have
\[\int_{\Omega} \rho_1 \beta=\lim_h \int_{\Omega} \rho_h \beta_h^{-} = \lim_h h\int_{\Omega} \rho_h  \vert \log(\min (1, \rho_h))\vert=0, \]
hence $\beta \rho_1=0$. Thanks to \pref{tvsupportf}, we obviously have $J(\rho_1)\ge \int_{\Omega} \dive(z) \rho_1$ (since $\dive(z)\in L^{\infty}$,  $\dive(z)\in \Gamma_d$), for the converse inequality, it is enough to observe that
\[J(\rho_1)\le \liminf_h J(\rho_h)=\liminf_h \int_{\Omega} \dive(z_h) \rho_h\]
and that $\dive(z_h)=-\frac{\varphi_h}{\tau}-\beta_h^++ \beta_h^{-}$ converges to $\dive(z)$ weakly in $L^q$ for every $q\in [1, +\infty)$. Since $\rho_h$ converges strongly to $\rho_1$ in $L^q$ when $q\in [1, \frac{d}{d-1})$ we deduce that $J(\rho_1)=\int_{\Omega} \dive(z) \rho_1$ which completes the proof of \pref{slackjko}.

\end{proof}


A first consequence of  the high integrability of $\dive(z)$ is that one can give a meaning to $z \cdot \nabla u$ for any $u \in \BV(\Omega)$.  Indeed, if $q\in [\frac{d}{d-1}, +\infty]$ and $q'$ denotes its conjugate exponent,  following Anzellotti \cite{Anzellotti},  if $u\in \BV(\Omega)\cap L^q(\Omega)$ and  $\sigma \in L^{\infty}(\Omega, \R^d)$ is such that $\dive(\sigma) \in L^{q'}(\Omega)$, one can define the distribution $\sigma \cdot Du$ by
\[\langle \sigma \cdot Du,  v \rangle =- \int_{\Omega}  \dive(\sigma) \; u  v  - \int_\Omega u \; \sigma \cdot \nabla  v, \; \forall v\in C_c^1(\Omega).\]
Then $\sigma \cdot Du$ is a Radon measure which satisfies  $\vert \sigma \cdot Du \vert \le \Vert \sigma \Vert_{L^{\infty}} \vert Du \vert$ (in the sense of measures) hence  is absolutely continuous with respect to $\vert D u\vert$.  Moreover one can also define a weak notion of normal trace of $\sigma$, $\sigma \cdot \nu \in L^{\infty}(\partial \Omega)$ such that the following integration by parts formula holds
\[\int_{\Omega} \sigma \cdot Du=-\int_{\Omega} \dive(\sigma) u+ \int_{\partial \Omega} u (\sigma \cdot \nu).\]
We refer to \cite{Anzellotti} for proofs. These considerations of course apply to $\sigma=z$ and $u=\rho_1\in \BV(\Omega)$ and in particular enable one to see $z\cdot D\rho_1$ as a measure and to interpret the optimality condition $J(\rho_1)=\int_{\Omega} \dive(z) \rho_1$ as $\vert D \rho_1 \vert=-z \cdot D \rho_1$ in the sense of measures. Finally, the fact that $\dive(z)\in L^\infty$  in  Theorem \ref{eulertvjko}  and the theory of variational mean curvature (see Tamanini  \cite{tamanini1982boundaries}, Massari  \cite{zbMATH03477953,massari1975frontiere}, Theorem 3.6 of  Gonzalez and Massari \cite{massari1994variational}) allows for conclusions about the regularity of the level sets, $F_t = \{\rho_1>t\}$  of $\rho_1$, the solution of \eqref{onestepjko}, we do not elaborate this regularity (which, anyway, only holds for fixed time step $\tau>0$) further here.

\section{Maximum and minimum principles for JKO steps}\label{sec-max}

Throughout this section, we further assume that $\Omega$ is a convex open bounded subset of $\R^d$, our aim is to establish bounds on the TV-JKO iterates given by \pref{onestepjko}. Since, the TV-JKO scheme aims at minimizing total variation at the fastest rate in the Wasserstein metric, it is natural to wonder whether when the initial condition is bounded from above and from below then the JKO-iterates remain so (with the same bounds). We shall answer affirmatively for the upper bound (maximum principle), as for the propagation of the lower bound (minimum principle), we  have been able to prove it only in special cases (dimension one and radially symmetric setting). 

 \subsection{Convexity along generalized geodesics}\label{subsec-prel}

Our aim is to deduce some bounds on $\rho_1$ from bounds on $\rho_0$. To do so, we shall combine some convexity arguments and a remarkable $\BV$ estimate due to De Philippis et al. \cite{de2016bv}. First we recall the notion of generalized geodesic from Ambrosio, Gigli and Savar\'e \cite{ambrosio2008gradient}. Given $\ovmu$, $\mu_0$ and $\mu_1$ in $\PPa$, and denoting by $T_0$ (respectively $T_1$) the optimal transport (Brenier) map between $\ovmu$ and $\mu_0$ (respectively $\mu_1)$, the \emph{generalized geodesic with base}  $\ovmu$ joining $\mu_0$ to $\mu_1$ is by definition the curve of measures:
\begin{equation}\label{defggeod}
\mu_t:=((1-t) T_0 +t T_1)_\# \ovmu, \; t\in [0,1]. 
\end{equation}
A key property of these curves introduced in \cite{ambrosio2008gradient} is the strong convexity of the squared distance estimate:
\begin{equation}\label{contw2}
W_2^2(\ovmu, \mu_t) \le (1-t) W_2^2(\ovmu, \mu_0)+ t  W_2^2(\ovmu, \mu_1)-t(1-t)  W_2^2( \mu_0, \mu_1).
\end{equation}
It is well-known that if $G$ : $\R_+\to \R \cup\{+\infty\}$ is a proper convex lower semi-continuous (l.s.c.) internal energy density, bounded from below such that $G(0)=0$ and which satisfies McCann's condition (see \cite{mcc97}) 
\begin{equation}\label{mccond}
\lambda \in \R_+ \to \lambda^d G(\lambda^{-d}) \mbox{ is convex nonincreasing}
\end{equation}
then defining the generalized geodesic curve $(\mu_t)_{t\in [0, 1]}$ by \pref{defggeod}, one has
\begin{equation}\label{geodconvex}
\int_{\Omega} G(\mu_t(x)) \mbox{d}x \le (1-t) \int_{\Omega} G(\mu_0(x)) \mbox{d} x+ t \int_{\Omega} G(\mu_1(x)) \mbox{d} x.
\end{equation}
In particular $L^p$ and uniform bounds are stable along generalized  geodesics:
\begin{equation}\label{boundsconvex}
\Vert \mu_t \Vert_{L^p}^p \le (1-t) \Vert \mu_0\Vert_{L^p}^p+ t \Vert \mu_0\Vert_{L^p}^p, \; \Vert \mu_t\Vert_{L^\infty} \le \max (\Vert \mu_0 \Vert_{L^\infty}, \Vert \mu_1\Vert_{L^\infty}),
\end{equation}
and 
\begin{equation}\label{entropybound}
\int_{\Omega} \mu_t(x)  \log(\mu_t(x)) \mbox{d}x \le (1-t) \int_{\Omega}  \mu_0(x) \log(\mu_0(x)) \mbox{d}x + t \int_{\Omega}  \mu_1(x) \log(\mu_1(x)) \mbox{d}x 
\end{equation}

An immediate consequence of \pref{contw2} (see chapter 4 of \cite{ambrosio2008gradient} for general contraction estimates)  is the following

\begin{lem}\label{distcont}
Let $K$ be a nonempty subset of $\PPa$, let $\mu_0 \in K$, $\mu_1\in \PPa$, if $\hat{\mu}_1 \in \argmin_{\mu \in K} W^2_2(\mu_1, \mu)$ is a Wasserstein projection of $\mu_1$ onto $K$, and if the generalized geodesic with base $\mu_1$ joining $\mu_0$ to $\hat{\mu}_1$ remains in $K$ then
\begin{equation}
W_2^2(\mu_0, \hat{\mu}_1) \le W_2^2(\mu_0, \mu_1)- W_2^2(\mu_1, \hat{\mu}_1). 
\end{equation}
\end{lem}

\begin{proof}
Since $\mu_t\in K$ we have $W^2_2(\mu_1, \hat{\mu}_1) \le W_2^2(\mu_1, \mu_t)$, applying  \pref{contw2} to the generalized geodesics with base $\mu_1$ joining $\mu_0$ to $\hat{\mu}_1$ we thus get
\[(1-t) W_2^2(\mu_1, \hat{\mu}_1) \le (1-t) W_2^2(\mu_1, \mu_0)-t(1-t) W_2^2(\mu_0, \hat{\mu}_1),\]
dividing by $(1-t)$ and then taking $t=1$ therefore gives the desired result. 
 
 \end{proof}

 The other result we shall use to derive bounds is a $\BV$ estimate of De Philippis et al. \cite{de2016bv}, which states that, given, $\mu \in \PPa \cap \BV(\Omega)$, and $G$ : $\R_+ \to \R\cup\{+\infty\}$, proper convex l.s.c., the solution of
 \begin{equation}\label{proxG}
 \inf_{\rho\in \PPa}  \Big \{\frac{1}{2} W_2^2(\mu, \rho)+ \int_{\Omega} G(\rho(x)) \mbox{d}x  \Big\}
 \end{equation}
 is $\BV$ with the bound
 \begin{equation}
 J(\rho) \le J(\mu).
 \end{equation}
 Taking in particular, 
 \[G(\rho):=\begin{cases} 0 \mbox{ if $\rho \le M$}, \\ + \infty \mbox{ otherwise, } \end{cases}\]
 this implies that the Wasserstein projection of $\mu$ onto the set defined by the constraint $\rho \le M$ has a smaller total variation than $\mu$.

 \subsection{Maximum  principle}\label{subsec-max}

 \begin{thm}\label{unifbound}
 Let $\rho_0 \in \PPa \cap L^\infty(\Omega)$ and let $\rho_1$ be the solution of \pref{onestepjko}, then $\rho_1 \in L^{\infty} (\Omega)$  with
 \begin{equation}\label{contraclp}
 \Vert \rho_1 \Vert_{L^\infty(\Omega)} \le \Vert \rho_0 \Vert_{L^\infty(\Omega)}. 
 \end{equation}
 
 \end{thm}
 
 \begin{proof}
 Thanks to \pref{boundsconvex} the set $K:=\{\rho \in \PPa  : \;  \rho \le \Vert \rho_0\Vert_{L^\infty(\Omega)} \mbox{ a.e.} \}$ has the property that the generalized geodesics (with any base) joining two of its points remains in $K$. Let then $\hat{\rho}_1$ be the $W_2$ projection of $\rho_1$ onto $K$ i.e. the solution of $\inf_{\rho\in K} W^2_2(\rho_1, \rho)$. Thanks to Lemma \ref{distcont} we have $W^2_2(\rho_0, \hat{\rho}_1)\le W_2^2(\rho_0, \rho_1)-W_2^2(\rho_1, \hat{\rho}_1)$ and thanks to Theorem 1.1 of \cite{de2016bv}, $J(\hat{\rho}_1) \le J(\rho_1)$. The optimality of $\rho_1$ for \pref{onestepjko} therefore implies $W_2(\rho_1, \hat{\rho}_1)=0$ i.e. $\rho_1 \le  \Vert \rho_0\Vert_{L^\infty(\Omega)}$.

  \end{proof}

 \begin{rem}\label{pmaxavecent}
 In section \ref{sec-opti}, we have used an approximation of \pref{onestepjko} with an additional small entropy term, the same bound as in Theorem  \ref{unifbound} will remain valid in this case. Indeed, consider a proper convex l.s.c.  and bounded from below internal energy density $G$ and consider given $h\ge 0$, the variant of \pref{onestepjko}
 \begin{equation}\label{onestepjkowithpenalty}
 \inf_{\rho \in \PPa}  \Big\{ \frac{1}{2\tau} W_2^2(\rho_0, \rho)+ J(\rho)+h \int_{\Omega} G(\rho(x)) \mbox{d} x \Big\}.
  \end{equation}
 Then we claim that the solution $\rho_h$ still satisfies $\rho_h \le \Vert \rho_0\Vert_{L^\infty(\Omega)}$. Indeed we have seen in the previous proof that the  Wasserstein projection $\hat{\rho}_h$ of $\rho_h$ onto the constraint $\rho \le \Vert \rho_0\Vert_{L^\infty(\Omega)}$ both diminishes $J$ and the Wasserstein distance to $\rho_0$. It turns out that it also diminishes the internal energy. Indeed, thanks to Proposition 5.2 of \cite{de2016bv}, there is  a measurable set $A$ such that $\hat{\rho}_h =\chi_A \rho_h + \chi_{\Omega \setminus A} \Vert \rho_0\Vert_{L^{\infty}}$, it thus follows that
$|\Omega\setminus A| \Vert \rho_0 \Vert_{L^\infty} = \int_{\Omega\setminus A}\rho_h$. So, 
 from the convexity of $G$ and Jensen's inequality,
 $$
 \int G(\hat \rho_h) = \int_{A} G(\rho_h) + |\Omega\setminus A| G\left(  |\Omega\setminus A|^{-1}   \int_{\Omega\setminus A} \rho_h \right)  \le \int G(\rho_h),
 $$ thus yielding the same conclusion as above. 
 
 \end{rem}

 \subsection{Minimum principle in special cases}\label{sec-minprinc}

 In dimension one, it turns out that we can obtain bounds from below by the same convexity arguments as for the maximum principle of Theorem \ref{unifbound}:

 \begin{prop}\label{boundbelow1d}
 Assume that $d=1$, that  $\Omega$ is a bounded interval and that $\rho_0 \ge \alpha >0$ a.e. on $\Omega$ then the solution $\rho_1$ of \pref{onestepjko} also satifies $\rho_1 \ge \alpha >0$ a.e. on $\Omega$. 
 
 \end{prop}
 
 \begin{proof}
 The proof is similar to that of Theorem \ref{unifbound} but using the Wasserstein projection on the set $K:=\{ \rho \in \PPa \; : \; \rho \ge \alpha\}$, the only thing to check to be able to use Lemma \ref{distcont} is that for any basepoint $\ovmu$ and any $\mu_0$ and $\mu_1$ in $K$, the generalized geodesic with base point $\ovmu$ joining $\mu_0$ to $\mu_1$ remains in $K$. The optimal transport maps $T_0$ and $T_1$ from $\ovmu$ to $\mu_0$ and $\mu_1$ respectively are nondecreasing and continuous and setting $T_t:=(1-t)T_0+t T_1$, one has 
 \[\ovmu=\mu_t(T_t)T'_t= \mu_0(T_0)T'_0=\mu_1 T'_1=(1-t) \mu_0(T_0)T'_0+ t  \mu_1(T_1)T'_1\ge \alpha T'_t\]
 which is easily seen to imply that $\mu_t \ge \alpha$ a.e..
 \end{proof}

 As a consequence of the previous minimum principle, integrating the Euler-Lagrange equation one can deduce higher regularity for the dual variable $z$:

\begin{coro}\label{euler1d}
Assume that $d=1$ and $\Omega$ is a bounded interval. If $\rho_1$ solves \pref{onestepjko} and $z$ is as in Theorem \ref{eulertvjko} then $z\in W^{1, \infty}_0(\Omega)$. If in addition $\rho_0\ge \alpha>0$ a.e. on $\Omega$, then $z\in W^{3, \infty}(\Omega)$. 
\end{coro}

\begin{proof}
The first claim is obvious because both $\varphi$ and $\beta$ ($\varphi$, $\beta$ and $z$ are as in Theorem \ref{eulertvjko}) are bounded hence so is $z'$. As for the second one when $\rho_0\ge \alpha>0$, thanks to Proposition \ref{boundbelow1d}, we also have $\rho_1\ge \alpha$ hence $\beta=0$ in \pref{eulerjko} and in this case $\dive(z)=z'=-\frac{\varphi}{\tau}$ is Lipschitz i.e. $z\in W^{2, \infty}$. One can actually go one step further because $x-\varphi'(x)=T(x)$ where $T$ is the optimal (monotone) transport between $\rho_1$ and $\rho_0$. This map is explicit in terms of the cumulative distribution function of $\rho_1$, $F_1$, and $F_0^{-1}$ the inverse of $F_0$, the cumulative distribution function of $\rho_0$, namely $T=F_0^{-1} \circ F_1$. But $F_1$ is Lipschitz since its derivative is $\rho_1$ which is $\BV$ hence bounded and $F_0^{-1}$ is Lipschitz as well since $\rho_0\ge \alpha>0$. This gives that $\varphi \in W^{2, \infty}$ hence  $z\in W^{3, \infty}$.
\end{proof}

 The proof of  Proposition \ref{boundbelow1d} unfortunately does not generalize to higher dimensions, because densities which are bounded from below by $\alpha$ are not stable by generalized geodesics. In the radially symmetric case, we can use the Euler-Lagrange equation to derive a minimum principle. We believe that JKO steps preserve lower bounds in more general situations but have not been able to prove it.

  \begin{prop}\label{boundbelowradial}
 Assume that  $\Omega=B(0,R)$ is the ball centered at $0$  or radius $R>0$ in $\R^d$, and that  $\rho_0$ is radially symmetric with $\rho_0 \ge \alpha >0$ a.e. on $\Omega$ then the solution $\rho_1$ of \pref{onestepjko} also satifies $\rho_1 \ge \alpha >0$ a.e. on $\Omega$. 
 
 \end{prop}
 
 \begin{proof}

 Let us write $\rho_0(x)=\tilr_0(r)$ with $r=\vert x \vert\in [0,R]$, since \pref{onestepjko} is invariant by rotation and strictly convex, it is easy to see that its unique solution $\rho_1$ is also radially symmetric, let us write it as  $\rho_1(x)=\tilr_1(r)$. Denoting by $c_d$ the $(d-1)$-Hausdorff measure of the unit sphere $S^{d-1}$,  and setting $\tilm_0:=c_d r^{d-1} \tilr_0$, $\tilm_1:=c_d r^{d-1} \tilr_1$,  observe that $\tilr_1$ is the minimizer of the one-dimensional convex functional
 \[\Fr(\tilr):=\frac{1}{2\tau} W_2^2(\tilm_0, c_d r^{d-1} \tilr)+c_d \int_0^R r^{d-1} \vert D \tilr \vert\]
 among nonnegative densities  $\tilr$ on $(0,R)$ such that $c_d \int_0^R r^{d-1} \tilr=1$ and $r^{d-1} D \tilr$ is a bounded Radon measure on $(0,R)$. Arguing as in the proof of Theorem \ref{eulertvjko}, the minimizer $\tilr_1$ is characterized by the Euler-Lagrange equation
 \begin{equation}\label{elrad}
 (\tilz r^{d-1})' + \frac{\tilfi}{\tau} r^{d-1}=\tilb \ge 0, \; \tilb \in L^{\infty}(0, R), \; \tilb \tilr_1=0,
  \end{equation}
 where $\tilfi$ is a Kantorovich potential between $\tilm_1$ and $\tilm_0$ and $\tilz  \in L^{\infty}(0, R)$ is such that
 \begin{equation}\label{elrad2}
 \vert \tilz\vert \le 1 \mbox{ a.e. and  } \int_0^R r^{d-1} \vert D \tilr_1 \vert=\int_0^R (\tilz r^{d-1})' \tilr_1.  
 \end{equation}
 Note that \pref{elrad} implies that $r^{d-1}\tilz$ is Lipschitz so that $\tilz$ is locally Lipschitz and 
 \begin{equation}\label{elrad3}
 \int_0^R r^{d-1} \vert D \tilr_1 \vert=-\int_0^R  r^{d-1} \tilz D \tilr_1
 \end{equation}
  Since $\tilr_1\in \BV_{\rm{loc}}(0,R)$,  we can perform a Hahn-Jordan decomposition of  $D\tilr_1$:
\begin{equation}\label{jordan}
D\tilr_1=\nu^+-\nu^-, \; \nu^+\ge 0, \; \nu^{-}\ge 0, \; \nu^+ \perp \nu^-,
\end{equation}
and set 
\begin{equation}\label{defapm}
A:=\spt(\vert D\tilr_1 \vert)=A^+\cup A^- \mbox{ with } A^+:=\spt(\nu^+), \; A^-:=\spt(\nu^-).
\end{equation}
Next, we observe that, using \pref{elrad3},  we have $\vert D \tilr_1\vert=\nu^+ +\nu^-= -\tilz(\nu^+-\nu^-)$, we thus deduce  that $\tilz=-1=\min \tilz $ $\nu^+$-a.e and since $\tilz$ is continuous we  actually have $\tilz=-1$ on $A^+=\spt(\nu^+)$. In a similar way, $\tilz=1=\max \tilz $ on $A^-:=\spt(\nu^-)$. 

\smallskip

Now let us show that $\tilr_1\ge \alpha$.  Assume, by contradiction, that the set where $\tilr_1<\alpha$ has positive measure in $(0,R)$, and let $r_0\in (0,R)$ be a continuity point of $\tilr_1$ such that $\tilr_1(r_0)<\alpha$, define then 
\[\begin{split}
a_-:= \inf\{ r\in (0, r_0) \; : \; \tilr_1 \le \alpha \mbox{ on $[r, r_0]$}\},\\
 a_+:=  \sup\{ r\in (r_0, R) \; : \; \tilr_1 \le \alpha \mbox{ on $[r_0, r]$}\}.
 \end{split}
 \]
We then have $0\le a_-<a_+ \le R$. Let us assume that  $a_->0$, we claim then that $a_-\in A^{-}$ since otherwise, $\tilr_1$ would be nondecreasing in a neighbourhood of $a_-$ which would imply $\tilr_1(a_--\eps)\le \alpha$ for small $\eps>0$, contradicting the definition of $a_-$, we thus have $\tilz(a_-)=1$. Since $\tilr_1$ is $\BV$ in a neigbourhood of $a_-$, it has a right and a left limit at $a_-$, again by minimality of $a_-$, the left limit of $\tilr_1$ at $a_-$ cannot be strictly smaller than $\alpha$, so there is an $\eps>0$ such that $\tilr_1>0$ on $I_-:=[a_- -\eps, a_-)$. Hence on $I_-$, \pref{elrad} becomes
\begin{equation}\label{elrad4}
\tilz '+\frac{d-1}{r} \tilz + \frac{\tilfi}{\tau}=0, 
\end{equation}
moreover, on $I_-$, $\tilfi$ is actually of class $C^1$ with $\tilfi'(r)=r-\tilT(r)$ where  $\tilT$ is the  (continuous) optimal transport between $\tilm_1$ and $\tilm_0$  obtained by the relation $F_{\tilm_0} \circ \tilT=F_{\tilm_1}$ (where $F_{\tilm_i}$ is the cumulative distribution function of $\tilm_i$ for $i=0, 1$). One can therefore differentiate \pref{elrad4} on $I_-$ so as to obtain
\begin{equation}\label{elrad5}
\tilz''+\frac{d-1}{r} \tilz '-\frac{(d-1)}{r^2} \tilz(r) + \frac{r-\tilT(r)}{\tau}=0, \forall r\in I_-.
\end{equation}
Since $\tilz$ is maximal at $a_-$, we first have
\[\lim_{\delta\to 0^+} \delta^{-1} [\tilz(a_-)  -\tilz(a_-  -\delta)]= -\frac{(d-1)\tilz(a_-)}{a_-}-\frac{\tilfi(a_-)}{\tau} \ge 0\]
but recalling \pref{elrad} we also have 
\[\begin{split}
0 &\ge \limsup_{\delta\to 0^+} \delta^{-1} [\tilz(a_-+\delta)-\tilz(a_-)]\\
&\ge   \lim_{\delta\to 0^+}  \delta^{-1} \int_{a_-}^{a_-+\delta} [-(d-1) s^{-1} \tilz(s)-\tau^{-1} \tilfi(s)] \mbox{d}s \\
&= -\frac{(d-1)\tilz(a_-)}{a_-}-\frac{\tilfi(a_-)}{\tau}
\end{split}\]
which shows that $\tilz$ is differentiable at $a_-$ with $\tilz '(a_-)=0$, this enables us to deduce that $\tilz''(a_-^-):=\lim_{\delta \to 0^+}  \tilz''(a_--\delta)\le 0$, with \pref{elrad5} this gives 
\[\tilT (a_-)-a_-=\tau \Big( \tilz''(a_-^-)-\frac{(d-1)}{a_-^2}\Big)  \le 0.\]
If $a_-=0$, since $\tilT(0)=0$, the same conclusion is reached with an equality. In a similar way, we obtain $\tilT(a_+)\ge a_+$ (again with an equality in case $a_+=R$). Using the fact that $\tilr_1 \le \alpha$ on $(a_-, a_+)$ (with strict inequality in a neighbourhood of $r_0$) together with $F_{\tilm_0} \circ \tilT=F_{\tilm_1}$ and $\tilr_0\ge \alpha$, we get
\[\begin{split}
\alpha c_d\frac{(a_+^d-a_-^d)}{d} &> F_{\tilm_1}(a_+)-F_{\tilm_1}(a_-) = F_{\tilm_0}( \tilT(a_+))- F_{\tilm_0}( \tilT(a_-))\\
& \ge  F_{\tilm_0}( a_+)-  F_{\tilm_0}( a_-) \ge \alpha c_d\frac{(a_+^d-a_-^d)}{d}
\end{split}\]
which yields the desired contradiction.

 \end{proof}
 
 Let us remark that the proof of Proposition \ref{boundbelowradial} gives an alternative proof of the minimum principle in dimension one.

\section{Convergence of the TV-JKO scheme under a lower bound estimate}\label{sec-conv} 
 
We are now interested in the convergence of the TV-JKO scheme to a solution of the fourth-order nonlinear equation \pref{pde4} as the time step $\tau$ goes to $0$. Throughout this section, we assume that $\Omega$ is a bounded open convex subset of $\R^d$ and that the initial condition $\rho_0$ satisfies 
\begin{equation}\label{boundalpha}
\rho_0 \in \PPa \cap BV(\Omega)\cap L^{\infty}(\Omega), \; \rho_0 \ge \alpha >0 \mbox{ a.e. on $\Omega$}.
\end{equation}
We fix a time horizon $T$, and for small $\tau>0$, define the sequence $\rho_k^\tau$ by
\begin{equation}\label{jkotv1}
\rho_0^\tau=\rho_0, \; \rho_{k+1}^\tau \in \argmin  \Big\{ \frac{1}{2\tau} W_2^2(\rho_{k}^\tau, \rho)+ J(\rho), \; \rho \in \BV\cap \PPa\Big\}
\end{equation}
for $k=0, \ldots N_\tau$ with $N_\tau:=[\frac{T}{\tau}]$.  Thanks to Theorem \ref{unifbound}, \pref{boundalpha} ensures that the JKO-iterates $\rho_k^{\tau}$ defined by \pref{jkotv1} also remain bounded $\rho_k^{\tau} \le \Vert \rho_0\Vert_{L^\infty(\Omega)}$.  We shall also assume that $\rho_k^\tau$ remains bounded from below by $\alpha$:
\begin{equation}\label{extraassumption}
\rho_k^\tau \ge \alpha>0 \mbox{ a.e. in $\Omega$, for every $k$ and $\tau$}
\end{equation}
which holds, as we have seen in subsection \ref{sec-minprinc} when $d=1$ or when $\Omega$ is a ball and $\rho_0$ is radially symmetric.

\smallskip

We extend this discrete sequence by piecewise constant interpolation i.e. 
\begin{equation}\label{extend}
\rho^{\tau}(t,x)=\rho_{k+1}^\tau(x), \; t \in (k\tau, (k+1)\tau], \; k=0, \ldots N_\tau, \; x\in \Omega.
\end{equation}

We shall see that $\rho^\tau$ converges  to a solution $\rho$ of
\begin{equation}\label{pde41}
\partial_t \rho + \dive \Big(\rho \;  \nabla  \dive  \Big( \frac{\nabla \rho}{\vert \nabla \rho \vert} \Big)  \Big)       =0, \; (t,x)\in (0,T)\times \Omega, \; \rho_{\vert_{t=0}}=\rho_0, 
\end{equation}
with the no-flux boundary condition
\begin{equation}\label{no-flux}
\rho \;  \nabla  \dive  \Big( \frac{\nabla \rho}{\vert \nabla \rho \vert} \Big) \cdot \nu=0 , \; \mbox{ on }  (0,T)\times \partial \Omega. 
\end{equation}
Let us introduce the spaces 
\[\begin{split}
H^1_{\dive}(\Omega):=\{z\in L^2(\Omega, \R^d) \; : \; \dive(z)\in L^2(\Omega)\},\\ 
H^2_{\dive}(\Omega):= \{z\in H^1_{\dive}(\Omega)  \; : \; \dive(z)\in H^1(\Omega)\}
\end{split}\]

Since $\rho$ is no more than $\BV$ in $x$, one has to be slightly cautious in the meaning of  $ \dive( \frac{\nabla \rho}{\vert \nabla \rho \vert}) $ which be conveniently done by interpreting this term as the negative of an element  in the subdifferential of $J$ (in the $L^2$ sense). For every $\rho \in \BV(\Omega)\cap L^2(\Omega)$ let us define
\[\partial J(\rho):=\{\dive(z) \; : \; z\in H^1_{\dive}(\Omega), \; \Vert z \Vert_{L^{\infty}} \le 1, \; z\cdot \nu=0 \mbox{ on $\partial \Omega$}, \;  J(\rho)=\int_{\Omega} \dive(z) \rho\}.\]
This leads to the following definition:
\begin{defi}
A weak solution of \pref{pde41}-\pref{no-flux} is a $\rho \in L^{\infty}((0,T), \BV(\Omega)\cap L^{\infty}(\Omega))\cap C^{0}([0,T], (\PP(\Omb), W_2)) $ such that there exists $z\in L^{\infty}((0,T)\times \Omega)\cap L^2((0,T), H^2_{\dive}(\Omega))$ with  
\begin{equation}\label{condcalibr}
\dive(z(t,.))\in \partial J(\rho(t,.))  \mbox{ for a.e. $t\in (0,T)$},
\end{equation}
and $\rho$ is a weak solution of
\begin{equation}\label{transportzxx}
\partial_t \rho -\dive \Big(\rho   \nabla \dive(z)\Big)   =0, \;  \rho_{\vert_{t=0}}=\rho_0, \; \rho \nabla \dive(z)\cdot \nu=0  \mbox{ on }  (0,T)\times \partial \Omega.
\end{equation}
i.e. for every $u \in C_c^{\infty}([0,T)\times \Omb)$ 
\[\int_0^T \int_\Omega (\partial_t  u \;  \rho-\rho \nabla \dive(z) \cdot \nabla u) \mbox{d}x \mbox{d}t=-\int_\Omega u(0,x) \rho_0(x) \mbox{d} x. \]

\end{defi}

We then have

\begin{thm}
If $\rho_0$ satisfies \pref{no-flux} and the JKO iterates $\rho_k^\tau$ obey the lower bound \pref{extraassumption}, there exists a vanishing sequence of time steps $\tau_n \to 0$ such that the sequence $\rho^{\tau_n}$ constructed by \pref{jkotv1}-\pref{extend} converges strongly in $L^p((0,T)\times (0,1))$ for any $p\in [1, +\infty)$ and in $L^{\infty}((0,T), (\PP(\Omb), W_2))$ to  a weak solution of \pref{pde41}-\pref{no-flux}.
\end{thm}

\begin{proof}
First, $\rho_0$ being $L^\infty$, we have a uniform $L^\infty$ bound on $\rho^\tau$ thanks to Theorem \ref{unifbound}, and from our extra lower bound assumption \pref{extraassumption} we have
\begin{equation}\label{unifbounds}
M:=\Vert \rho_0\Vert_{L^{\infty}} \ge    \rho^\tau(t,x) \ge \alpha, \; t\in [0,T], \; \mbox{ a.e. } x\in \Omega.
 \end{equation}

Moreover, by construction of the TV-JKO scheme \pref{jkotv1}, one has
\begin{equation}\label{estim0}
\frac{1}{2\tau} \sum_{k=0}^{N_\tau}  W_2^2(\rho_k^\tau, \rho_{k+1}^\tau) \le J(\rho_0), \;   \;  \sup_{t\in [0,T]} J(\rho^\tau(t,.)) \le J(\rho_0)
\end{equation}

 By using an Aubin-Lions type compactness Theorem of Savar\'e and Rossi (Theorem 2 in \cite{rossi2003tightness}), the fact that the embedding of $\BV(\Omega)$ into $L^p(\Omega)$ is compact for every $p\in [1, \frac{d}{d-1})$ as well as a refinement of Arz\`ela-Ascoli Theorem (Proposition 3.3.1 in \cite{ambrosio2008gradient}), one obtains (see  section 4 of \cite{DiF2014curves}  or section 5 of \cite{carlier2017split} for details) that, up to taking suitable sequence of vanishing times steps  $\tau_n\to 0$, we may assume that 
 \begin{equation}\label{cvrhotLp}
 \rho^\tau \to \rho \mbox{ a.e.  in $(0,T)\times \Omega$  and in } L^p((0,T)\times \Omega), \; \forall p\in [1, \frac{d}{d-1})
 \end{equation} 
 and 
 \begin{equation}\label{cvunifw}
\sup_{t\in [0,T]}  W_2(\rho^\tau(t,.), \rho(t, .)) \to 0 \mbox{ as } \tau \to 0, 
 \end{equation}
 for some limit curve $\rho \in C^{0,\frac{1}{2}} ([0,T], (\PP(\Omb), W_2)) \cap L^q((0,T)\times \Omega)$. From \pref{unifbounds} and Lebesgue's dominated convergence Theorem, we deduce that the convergence in \pref{cvrhotLp} actually holds for any $p\in [1, +\infty)$. It also follows from \pref{unifbounds} and \pref{estim0}, that $\rho \in L^{\infty}((0,T), \BV(\Omega)\cap L^{\infty}(\Omega))$ and that $\rho \ge \alpha$. 
 
  \smallskip
  
 We deduce from the fact that $\rho_k^\tau\ge \alpha>0$ and Theorem \ref{eulertvjko} that for each $k=0, \ldots, N_\tau$,  there exists $z_k^\tau  \in L^{\infty}(\Omega, \R^d)$ such that $\dive(z_k^\tau)  \in W^{1, \infty}(\Omega)$ and 
 \begin{equation}\label{eulerk1}
 \Vert z_k^\tau\Vert_{L^\infty} \le 1, \; z_k^\tau\cdot \nu=0 \mbox{ on $\partial \Omega$}, \; J(\rho_{k}^\tau)=\int_\Omega \dive (z_k^\tau) \rho_k^\tau, 
 \end{equation}
  and the optimal (backward) optimal transport $T_{k+1}^\tau$ from $\rho_{k+1}^\tau$ to $\rho_k^\tau$ is related to $z_{k+1}^\tau$ by
  \begin{equation}\label{eulerk2}
  \id-T_{k+1}^\tau=-\tau \nabla \dive ( z_{k+1}^\tau).
  \end{equation}
We extend $z_k^\tau$ in a piecewise constant way i.e. set
\begin{equation}\label{extendz}
z^\tau(t,x)=z_{k+1}^\tau(x), \; t\in (k\tau, (k+1)\tau], \, k=0, \ldots, N_\tau, \; x\in \Omega. 
\end{equation}
We then observe that 
\[\begin{split}
W_2^2(\rho_k^\tau, \rho_{k+1}^\tau) &=\int_\Omega \vert x-T_{k+1}^\tau(x)\vert^2 \rho_{k+1}^\tau(x) \mbox{d} x\\
&=\tau^2 \int_\Omega \vert \nabla \dive( z_{k+1}^\tau)\vert ^2 \rho_{k+1}^\tau(x) \mbox{d} x\\
 &\ge \alpha \tau^2 \int_\Omega \vert \nabla \dive( z_{k+1}^\tau)\vert ^2 \mbox{d} x
\end{split}\]
Thanks to \pref{estim0} we thus deduce that $\nabla \dive z^\tau$ is bounded in $L^2((0,T)\times \Omega)$, since $\dive(z^\tau)$ has zero-mean, with Poincar\'e-Wirtinger inequality, we obtain
\begin{equation}\label{h2bound}
 \int_0^T  \Vert \dive(z^\tau) \Vert^2_{H^1(\Omega)} \mbox{d}t \le C.
\end{equation}
We may therefore assume (up to further suitable extractions) that there is some $z\in L^{\infty}((0,T)\times \Omega)\cap L^2((0,T), H^2_{\dive}(\Omega))$ such that $z^\tau$ converges to $z$ weakly $*$ in $L^\infty((0,T)\times \Omega)$ and  $(\dive(z^\tau), \nabla \dive(z^\tau))$ converges  weakly in $L^2((0,T)\times \Omega)$ to $(\dive(z), \nabla \dive(z))$. Of course $\Vert z \Vert_{L^\infty} \le 1$ and $z(t,.) \cdot \nu=0$ on $\partial \Omega$ for a.e. $t$. Note also that $\rho^\tau \nabla \dive( z^\tau)$ converges weakly in $L^1((0,T)\times \Omega )$ to $\rho  \nabla \dive (z)$.

\smallskip

The limiting equation can now be derived using standard computations (see the proof of Theorem 5.1 of the seminal work \cite{jordan1998variational}, or chapter 8 of \cite{santambrogio2015optimal}):
Let $u\in C_c^2([0,T)\times \Omb)$ and observe that
\begin{equation*}
\int_0^T \int_\Omega \partial_t  u \; \rho^\tau \mathrm{d}x\mathrm{d}t = \sum_{k=1}^{N_\tau} \left( \int_\Omega u(k\tau,x) (\rho_k^\tau(x) -  \rho_{k+1}^\tau(x))\mathrm{d}x \right)- \int_\Omega u(0,x) \rho_1^\tau(x) \mathrm{d}x.
\end{equation*}
Recalling that $\rho_k^\tau = {T_{k+1}^\tau}_\# \rho_{k+1}^\tau$, and applying Taylor's theorem, we have 
\begin{align*}
 &\sum_{k=1}^{N_\tau } \left( \int_\Omega u(k\tau,x) (\rho_k^\tau(x) - \rho_{k+1}^\tau(x))\mathrm{d}x \right)\\
 &=
\sum_{k=1}^{N_\tau } \left( \int_\Omega ((T_{k+1}^\tau(x)-x) \cdot \nabla u(k\tau,x) + \tilde R_\tau(x)  )    \rho_{k+1}^\tau \mathrm{d}x \right)\\
&= \sum_{k=1}^{N_\tau } \left( \int_\Omega  ( \tau( \nabla \dive(z^\tau_{k+1}))\cdot \nabla u(k\tau,x)  + \tilde R_\tau (x)  )\rho_{k+1}^\tau \mathrm{d}x \right),
\end{align*}
where $\vert \tilde R_\tau(x) \vert \leq C \Vert  D^2 u (k\tau, \cdot)\Vert_{L^\infty} \vert T_{k+1}^\tau(x) - x\vert^2$.
Note also that  for $t \in (k\tau, (k+1)\tau]$, $\vert \nabla u(k\tau, \cdot) - \nabla u(t,\cdot) \vert \leq \tau \Vert \partial_t \nabla u\Vert_{L^\infty}$.
Therefore,
\begin{equation}\label{approxtransp}
\int_0^T \int_\Omega (\partial_t  u \;  \rho^\tau-\rho^\tau  \nabla \dive(z^\tau)\cdot \nabla u) \mbox{d}x \mbox{d}t=-\int_\Omega u(0,x) \rho_1^\tau(x) \mbox{d} x+R_\tau(u)
\end{equation}
with
\begin{equation}
\vert R_\tau(u) \vert \le C \max\{\Vert D^2 u \Vert_{L^\infty}, \Vert  \partial_t \nabla u \Vert_{L^\infty}\} \sum_{k=0}^{N_\tau}  W_2^2(\rho_k^\tau, \rho_{k+1}^\tau) \le C \tau.
\end{equation}
Passing to the limit $\tau$ to $0$ in \pref{approxtransp} yields that $\rho$ is a weak solution to 
\[\partial_t \rho -\dive \Big(\rho \nabla \dive (z)\Big)=0, \;  \rho_{\vert_{t=0}}=\rho_0, \; \rho \nabla \dive (z) \cdot \nu=0  \mbox{ on }  (0,T)\times\partial \Omega.\]
It remains to prove that $J(\rho(t,.))=\int_\Omega \dive(z(t,x)) \rho(t,x) \mbox{d}x$,  for a.e. $t\in (0,T)$. The inequality $J(\rho(t,.))\ge \int_\Omega \dive(z(t,x)) \rho(t,x) \mbox{d}x$ is obvious since $z(t,.) \in H^1_{\dive}(\Omega)$, $z(t,.)\cdot \nu=0$ on $\partial \Omega$ and  $\Vert z(t,.)\Vert_{L^{\infty}}\le 1$. To prove the converse inequality, we use Fatou's Lemma, the lower semi-continuity of $J$, \pref{eulerk1} and the weak-convergence in $L^1((0,T)\times \Omega)$ of $\rho^\tau  \dive( z^\tau)$ to $\rho  \dive( z)$:
\[\begin{split}
\int_0^T J(\rho(t,.)) \mbox{d} t & \le \int_0^T  \liminf_{\tau} J(\rho^\tau(t,.)) \mbox{d} t \\
&\le     \liminf_{\tau} \int_0^1  J(\rho^\tau(t,.)) \mbox{d} t   \\
&=  \liminf_{\tau} \int_0^T \int_\Omega \dive( z^\tau(t,x))  \rho^\tau(t,x) \mbox{d} x \mbox{d} t\\
&= \int_0^T \int_\Omega  \dive(z(t,x)) \rho(t,x) \mbox{d} x \mbox{d} t
\end{split}\]
which concludes the proof.

\end{proof}

\smallskip

{\textbf{Acknowledgements:}} The authors wish to thank Vincent Duval and Gabriel Peyr\'e for suggesting the TV-Wasserstein problem to them as well as for fruitful discussions. They also thank Maxime Laborde and Filippo Santambrogio for helpful remarks in particular regarding the maximum principle.

\bibliographystyle{plain}

\bibliography{mybiblio}

\end{document}